\documentclass[reqno,12pt]{amsart} 
\usepackage{amsmath,amsthm} 
\usepackage{amssymb} 
\newtheorem{Theorem}{Theorem } 
\newtheorem{lemma}{Lemma}

\font\ff=cmsy10 
\def\tiF{\text{\ff F\kern 0pt}{\;}^{ -1}} 
\def\tF{\text{\ff F\kern 0pt}} 
\begin{document} 
\title[]{On the support of solutions to the \\Zakharov-Kuznetsov equation} 
\author{Eddye Bustamante, Pedro Isaza and Jorge Mej\'{\i}a}  
\thanks{With the support of  Colciencias, Colombia, Project 574-2009} 
  
\address{Eddye Bustamante M., Pedro Isaza J., Jorge Mej\'{\i}a L. \newline  
Departamento de Matem\'aticas\\Universidad Nacional de Colombia\newline  
A. A. 3840 Medell\'{\i}n, Colombia} 
\email{eabusta0@unal.edu.co, pisaza@unal.edu.co, jemejia@unal.edu.co}

\begin{abstract} 
In this article we prove that sufficiently smooth solutions of the Zakharov-Kuznetsov equation:
\[ 
\partial _{t}u+\partial _{x}^{3}u+\partial_x\partial^2_yu+u\partial _{x}u =0 \,,
\]
that have compact support for two different times are identically zero.
\end{abstract} 
\maketitle

\section{Introduction}
In this article we consider the Zakharov-Kuznetsov equation
\begin{equation}
\partial_{t}u+\partial_{x}^{3}u+\partial_{x}\partial_{y}^{2}u+u\partial_{x}u=0 \text{,}\quad (x,y)\in\mathbb R^2, \;t\in[0,1]\;.\label{ZK}
\end{equation}
Equation \eqref{ZK} is a bidimensional generalization of the Korteweg-de Vries (KdV) equation which is a mathematical model to describe the propagation of nonlinear
ion-acoustic waves in magnetized plasma (\cite{ZaKu}).

Our goal in this article is to prove that a sufficiently smooth solution $u=u(x,y,t)$ of \eqref{ZK} which has compact support at two different times must vanish identically. Results concerning local and global well-posedness for the Cauchy problem associated to equation \eqref{ZK} can be found in the articles, \cite{F}, \cite{BL}, \cite{LP1}, \cite{LPS}, and \cite{LP2}.

In \cite{SS}, Saut and Scheurer proved a result concerning a general class of dispersive-dissipative equations, including the KdV equation, which afirms that if a sufficiently smooth solution $u=u(x,t)\,,\;x\in\mathbb{R}^n\,,\;t\in\mathbb{R}$, of this type of equation, vanishes in a nonempty open set of $\mathbb{R}^n\times\mathbb{R}$, then it is identically zero.\par
 
Kenig, Ponce and Vega in \cite{KPV} proved that if a sufficiently smooth solution $u$ of the KdV equation is such that for some $B\in\mathbb{R}$, and two different times $t=0$ and $t=1$,
\begin{equation}\text{supp}\;u(\cdot,0)\,,\;\;\text{supp}\;u(\cdot,1)\subset(-\infty,B]\,,\label{soporte}\end{equation}
then $u\equiv 0$. First of all, they observed that with this condition on the support at time $t=0$, the solution presents exponential decay to the right $(x>0)$ for every $t>0$, which enables the use of a Carleman type estimate in order to show that the solution is zero in a half-strip $[R,+\infty)\times [0,1]$. In particular, $u$ vanishes in a nonempty open set of $\mathbb{R}\times[0,1]$, which permits to apply Saut-Scheurer's result to conclude that $u\equiv 0$.\par
Using refinements of the method in \cite{KPV}, unique continuation principles have been successively improved for the KdV and Schr\"odinger equations (see for example \cite{EKPV1} and \cite{EKPV2}).\par
In \cite{B}, Bourgain introduced an approach, based on Complex Analysis methods, to prove that if sufficiently smooth solutions of certain dispersive equations, including the KdV equation, are compactly supported on a nontrivial time interval, then they are identically zero.\par
Although the result in \cite{B} is weaker, in the KdV case, than that in \cite{SS}, unlike Saut and Scheurer's result, Bourgain's result can be obtained for the Zakharov-Kuznetsov equation. In fact, Panthee in \cite{P} proved the following result:

\begin{Theorem}\label{Panthee} Let $u\in C([0,1];H^4(\mathbb{R}^2))$ be a solution of  equation \eqref{ZK} such that for some $B>0$
\begin{equation}\text{supp}\;u(t)\subset[-B,B]\times[-B,B]\quad\forall t\in[0,1]\;.\label{pan}\end{equation}
Then $u\equiv 0$.
\end{Theorem}
In our work we will only require condition \eqref{pan} for two different times. More precisely, we prove the following result.

\begin{Theorem}\label{main} 
Let $u\in C([0,1];H^{4}(\mathbb{R}^{2}))\, \cap \,C^{1}([0,1];L^{2}(\mathbb{R}^{2}) )$ be a solution of  \eqref{ZK} such that, for some $B>0$,
\[
supp~u( 0),\;supp~u( 1)  \subseteq [-B,B]\times [-B,B]. 
\]
Then, $u\equiv 0.$
\end{Theorem}
The proof of theorem \ref{main} follows the ideas of Kenig, Ponce and Vega in \cite{KPV}. In first place, we observe that if the solutions of the Zakharov-Kuznetsov equation have exponential decay for $x>0$ and $y\in\mathbb{R}$ at time $t=0$, and exponential decay for $x<0$ and $y\in\mathbb{R}$ at time $t=1$, then these solutions have exponential decay as $x^2+y^2$ goes to infinity at all times $t\in[0,1]$. This fact allows us to use a Carleman estimate of $L^2-L^2$ type, in order to establish that for the function $u$ in Theorem \ref{main} there exists $B>0$ such that $\text{supp}\;u(t)\subset[-B,B]\times[-B,B]$ for all $t\in[0,1]$. In this manner, by Theorem \ref{Panthee}, $u\equiv 0$.\par 

From now on, we will say that $f \in H^k (e^{2 \beta x} dx dy)$ if $\partial ^\alpha f \in L^2(e^{2\beta x}dx dy)$ for all multi-index $\alpha=(\alpha_1,\alpha_2)$ with $|\alpha|\leq k$. In a similar way we define $H^k (e^{2 \beta x} e^{2 \beta y}dx dy)$.

The decay property of the solutions of the Zakharov-Kuznetsov equation, mentioned before, plays a central role in this article and it is proved in the following theorem:

\begin{Theorem}\label{decay}
Let $u\in C( [0,1];H^{4}( \mathbb{R}^{2}) ) \cap C^{1}( [0,1];L^{2}( \mathbb{R}^{2}) ) $ be a solution of \eqref{ZK} such that for all $\beta >0$, $u( 0) \in L^{2}( e^{2\beta x}e^{2\beta |y|}dxdy)$ and $u( 1) \in L^{2}( e^{-2\beta x}e^{2\beta |y|}dxdy$). Then $u$ is a bounded function from $[0,1]$ with values in $H^{3}( e^{2\beta | x|}e^{2\beta |y| }dxdy) $ for all $\beta >0$.
\end{Theorem}

The Carleman's type estimates are proved in the following theorem:
\begin{Theorem}\label{Carleman}
Let $w\in C([0,1];H^{3}(\mathbb{R}^{2}))\,\cap\, C^{1}([0,1];L^{2}(\mathbb{R}^{2}))$, be a function such that for all $\beta >0$

(i)\  $w$ is bounded from $[0,1]$ with values in $H^{3}( e^{2\beta |x|}e^{2\beta |y|}dxdy)$, and 

(ii) $w'\in L^{1}([0,1];L^{2}( e^{2\beta |x|}e^{2\beta |y|}dxdy))$. 

Then, for all $\lambda \neq 0$,
\begin{equation}
\|e^{\lambda x}w\|\leq \|e^{\lambda x}w(0) \| _{L^{2}(\mathbb{R}^{2}) }+\| e^{\lambda x}w(1) \|_{L^{2}(\mathbb{R}^{2})}+\| e^{\lambda x}(w'+\partial_{x}^{3}w+\partial _{x}\partial _{y}^{2}w) \| \,,\label{EstimativoCarleman}
\end{equation}
where
$\|\cdot\|:=\|\cdot\|_{L^{2}(\mathbb{R}^{2}\times[0,1])}$.

A similar estimate also holds with $y$ instead of $x$ in the exponents.
\end{Theorem}
Our proof of \eqref{EstimativoCarleman} relies only on the Fourier transform in the space variables and on the elementary properties of absolutely continuous functions in the time variable.

The paper is organized as follows: in section $2$ we prove Theorem \ref{decay} and in section $3$ we prove Theorem \ref{Carleman}. Finally, in section $4$, using Theorem \ref{decay}, Theorem \ref{Carleman} and Theorem \ref{Panthee}, we establish Theorem \ref{main}.

Throughout this article the letter $C$ will denote diverse positive constants which may change from line to line and depend on parameters which are clearly established in each case.

\section{Apriori estimates (Proof of Theorem \ref{decay})}

The proof of Theorem \ref{decay} is based on the following lemmas.

The first lemma is an interpolation result which can be proved using the Three-line Theorem:
\begin{lemma} \label{interpolation}
For $s>0$ and $\beta>0$ let $f\in H^s(\mathbb R^2) \cap L^2(e^{2\beta x }dx dy)$. Then, for $\theta \in[0,1]$:
\begin{equation}
\| J^{\theta s} (e^{(1-\theta) \beta x}f)\|_{L^2} \leq C \|J^s f \|_{L^2}^\theta \| e^{\beta x}f\|_{L^2}^{1-\theta}, \label{interpolation.1}
\end{equation}
where $[J^sf]\,\widehat{\;}\,(\xi):=(1+|\xi|^2)^{s/2} \widehat f (\xi)$ and $C=C(s,\beta)$.

(Here, $\widehat{\;}$ denotes the spatial Fourier transform in $\mathbb R^2$, and $\xi=(\xi_1,\xi_2)$, where $(\xi_1,\xi_2)$ are the variables in the frequency space corresponding to the space variables $(x,y)$).

Similarly, if $f\in H^s(\mathbb R^2) \cap L^2(e^{2(\beta x+\beta y) }dx dy)$. Then, for $\theta \in[0,1]$:
\begin{equation}
\| J^{\theta s} (e^{(1-\theta) (\beta x+\beta y)}f)\|_{L^2} \leq C \|J^s f \|_{L^2}^\theta \| e^{\beta x+\beta y}f\|_{L^2}^{1-\theta}. \label{interpolation.2}
\end{equation}
\end{lemma}

The exponential decay in Theorem \ref{decay} is obtained in two steps. In the first step we establish the boundedness of $u(t)$ in the space $H^3(e^{2\beta x}dx dy)$, and then, using this fact, we prove the boundedness of $u(t)$ in the space $H^3(e^{2\beta x+2\beta y}dx dy)$. The conclusion of the proof follows from the symmetry properties of the equation.

\begin{lemma}\label{derecha} 
Let $u\in C([0,1];H^{4}(\mathbb{R}^{2}) ) \cap C^{1}([0,1];L^{2}(\mathbb{R}^{2}) )$ be a solution of  $(\ref{ZK}) $ such that
for all $\beta>0$, $u(0) \in L^{2}( e^{2\beta x}dxdy)$. Then $u$ is a bounded function from $[0,1]$ with values in $H^{3}( e^{2\beta x}dxdy) $ for all $\beta >0$.
\end{lemma}

\begin{proof} We will first prove that $t\mapsto u(t)$ is bounded from $[0,1]$ with values in $L^2(e^{2\beta x}dxdy)$. Let $\varphi\in C^{\infty}(\mathbb{R})$ be a decreasing function with $\varphi(x)=1$ if $x<1$ and $\varphi(x)=0$ if $x>10$. For $n\in\mathbb{N}$ we define
\[\phi_n(x):=e^{2\beta\theta_n(x)}\;,\]
where $\theta_n(x):=\int_0^x\varphi(\frac{x'}{n})dx'$.

It is easily seen that for every $n$, $\phi_n$ is an increasing function, $\phi_n(x)=e^{2\beta x}$ if $x\leq n$, and $\phi_n(x)\equiv d_n\leq e^{20\beta n}$ if $x>10n$. Also, $\phi_n\leq\phi_{n+1}$ for every $n$ and
\[|\phi_n^{(j)}(x)|\leq C_{j, \beta}\phi_n(x)\quad\forall j\in\mathbb{N}\quad\forall x\in\mathbb{R}\;.\]
Multiplying equation \eqref{ZK} by $u\phi_n$ and integrating by parts in $\mathbb{R}_{xy}^2$  we obtain:
\[
\frac{1}{2}\frac{d}{dt}\int u^{2}\phi _{n}+\frac{3}{2}\int ( \partial_{x}u)^{2}\phi_{n}'-\frac{1}{2}\int u^{2}\phi _{n}'''+\frac{1}{2}\int ( \partial_{y}u)^{2}\phi_{n}'-\frac{1}{3}\int u^{3}\phi _{n}'=0,
\]
Therefore, discarding positive terms and applying Sobolev imbeddings,
\begin{align*}\frac12\frac{d}{dt}\int u^2\phi_n&\leq \frac{1}{2}C_{3,\beta}\int u^2\phi_n+ \frac13\|u(t)\|_{L^{\infty}(\mathbb{R}^2)}C_{1,\beta}\int u^2\phi_n\\
&\leq (C_{3,\beta}+C\|u\|_{C([0,1];H^2)})\int u^2\phi_n\equiv C_{\beta,u}\int u^2\phi_n\;,
\end{align*}
and applying Gronwall's lemma and the Monotone Convergence Theorem with $n\to\infty$ we conclude that
\begin{equation}\int u(t)^2e^{2\beta x}dxdy\leq C\int u(0)^2e^{2\beta x}dxdy\quad\forall t\in[0,1]\;,\label{peso1}\end{equation}
which proves that $t\mapsto u(t)$ is bounded from $[0,1]$ with values in $L^2(e^{2\beta x}dxdy)$.\par

Since this boundedness holds for each $\beta>0$, and, on the other hand, $u\in C([0,1];H^4)$, we can apply the interpolation inequality (\ref{interpolation.1}) with $s=4$, $\theta=\frac{3}{4}$, to conclude that $t\mapsto u(t)$ is bounded from $[0,1]$ with values in $H^3(e^{2\beta x} dx dy)$, which completes the proof of Lemma \ref{derecha}.
\end{proof}

\begin{lemma}\label{derecha2}
Let $u\in C([0,1];H^4(\mathbb R^2))\cap C^1([0,1];L^2(\mathbb R^2))$ be a solution of \eqref{ZK}. If $u(0)\in L^2(e^{2\beta x}e^{ 2\beta |y|}dx dy)$ for all $\beta>0$, then $u$ is a bounded function from $[0,1]$ with values in $H^3(e^{2\beta x} e^{2\beta|y|}dx dy)$.
\end{lemma}

\begin{proof}
Our first step will be to prove that the $u$ is bounded from $[0,1]$ to $L^{2}( e^{2\beta x}e^{2\beta y}dxdy)$.

Since $u(0)\in L^2(e^{2\beta x} e^{2\beta |y|} dx dy)$, then $u(0)\in L^2(e^{2\beta x} dx dy)$, and in consequence, by Lemma \ref{derecha}, $u$ is bounded from $[0,1]$ with values in $H^3(e^{2\beta x}dxdy)$ for all $\beta>0$.

Let $w( t) :=e^{\beta x}u( t) $. Since $u$ is a solution of (\ref{ZK}), it follows that $w$ satisfies the equation 
\begin{equation}
e^{\beta x}u'-\beta ^{3}w+3\beta ^{2}\partial _{x}w-3\beta \partial_{x}^{2}w+\partial_{x}^{3}w-\beta \partial _{y}^{2}w+\partial _{x}\partial_{y}^{2}w-\beta uw +u\partial _{x}w =0\,.\label{w}
\end{equation}
Let us notice that, since $u(t)\in H^3(e^{2\beta x}dxdy)$, and $u$ satisfies equation \eqref{ZK}, all terms in the former equation belong to $L^2(\mathbb R^2)$.

For $n\in \mathbb N$ let us define $\phi_n(y):=e^{2\beta \theta_n(y)}$, where the function $\theta_n$ is the same function defined in the proof of Lemma \ref{derecha}.

Multiplying equation \eqref{w} by $w\phi_n(y)$ and integrating by parts in $ \mathbb{R}^{2}_{xy}$ we obtain:
\[
\int e^{\beta x}u'w\phi _{n}-\beta ^{3}\int w^{2}\phi_{n}+3\beta \int ( \partial_{x}w) ^{2}\phi _{n}+\beta \int( \partial _{y}w) ^{2}\phi _{n}-\frac{1}{2}\beta \int w^{2}\phi_{n}''
\]
\begin{equation}
+\int ( \partial _{y}w) ( \partial _{x}w) \phi_{n}'-\beta \int uw^{2}\phi _{n}-\frac{1}{2}\int w^{2}(\partial _{x}u) \phi _{n}=0.\label{anterior}
\end{equation}

For the first term we will see that $$t\mapsto \int\limits_{\mathbb{R}^{2}}e^{\beta x}u( t) w( t) \phi _{n}( y)dxdy=\int w^{2}\phi _{n}$$ is absolutely continuous in $[ 0,1] $ and that
\begin{equation}
\dfrac{1}{2}\dfrac{d}{dt}\int w^{2}\phi _{n}=\int e^{\beta x}u'w\phi _{n}\text{ } a.e.\text{ }t\in [ 0,1] . \label{9}
\end{equation}

In fact, since $u\in C^1([0,1];L^2(\mathbb R^2))$ and for $m\in \mathbb N$, $\phi_m(x) \phi_n(y) \in L^\infty(\mathbb R^2)$
\begin{equation*}
\frac{d}{dt}\langle u(t),\phi_m(\cdotp_x) \phi_n(\cdotp_y) u(t) \rangle =2 \int u'(t)\phi_m (x) \phi_n(y) u(t).
\end{equation*}

Thus the fundamental theorem of Integral Calculus implies that
\begin{equation*}
\int u(t) \phi_m(x) \phi_n(y) u(t)-\int u(0)\phi_m(x)\phi_n(y)u(0)=2 \int_0^t [\int u'(\tau)\phi_m(x)\phi_n(y)u(\tau)dx dy] d\tau.
\end{equation*}

An easy application of Dominated Convergence Theorem in the former equality gives, when $m$ goes to $\infty$, that
\begin{equation*}
\int u(t) e^{2\beta x}\phi_n(y) u(t)-\int u(0)e^{2\beta x}\phi_n(y)u(0)=2 \int_0^t [\int u'(\tau)e^{2\beta x}\phi_n(y)u(\tau)dx dy] d\tau.
\end{equation*}

which implies \eqref{9}.

Taking into account that $ | \phi _{n}'( y) |= | 2\beta \varphi(\frac{y}{n}) \phi _{n}(y) | \leq 2\beta \phi _{n}( y) $, from \eqref{anterior} and \eqref{9}, it follows that
\begin{align}
\frac12\frac{d}{dt}\int w^{2}\phi _{n} &\leq \beta ^{3}\int w^{2}\phi _{n}-\beta\int ( ( \partial _{x}w) ^{2}-2 | \partial_{x}w| | \partial _{y}w| +( \partial_{y}w) ^{2}) \phi _{n} +\frac12\beta C_{2,\beta }\int w^{2}\phi _{n}\notag \\ &+\beta C \| u\|_{C( [ 0,1] ;H^{2}( \mathbb{R}^{2}) ) }\int w^{2}\phi _{n} +C \| \partial _{x}u\| _{C( [ 0,1];H^{2}( \mathbb{R}^{2}) ) }\int w^{2}\phi _{n}\notag \\
&\equiv C_{\beta,u}\int w^{2}\phi _{n}-\beta\int \bigl( |\partial _{x}w| -|\partial _{y}w|\bigr)^{2}\phi_n\notag\\
&\leq C_{\beta,u}\int w^{2}\phi _{n}\;
\text{a.e. }t\in [ 0,1]\,,\label{exponencial}
\end{align}
which, as in Lemma \ref{derecha}, implies that $u$ is  bounded  from $[ 0,1] $ to $L^{2}( e^{2\beta x}e^{2\beta y}dxdy) $. This, together with the fact that $u\in C([0,1]; H^4)$ and the interpolation inequality (\ref{interpolation.2}) with $s=4$ and $\theta=\frac{3}{4}$, shows that $u$ is bounded from $[0,1]$ with values in $H^3(e^{2\beta x}e^{2\beta y}dxdy)$ for all $\beta>0$.

Finally, if we define $\widetilde u(x,y,t):=u(x,-y,t)$, then $\widetilde u$ is also a solution of (\ref{ZK}), with $\widetilde u(0)\in L^2(e^{2\beta x}e^{2\beta|y|}dxdy)$ and therefore $\widetilde u$ is bounded from $[0,1]$ with values in $H^3(e^{2\beta x}e^{2\beta y}dxdy)$ for all $\beta>0$, i.e. $u$ is bounded from $[0,1]$ with values in $H^3(e^{2\beta x}e^{-2\beta y}dxdy)$; which proves the lemma.

\end{proof}

The proof of Theorem \ref{decay} follows immediately from  Lemma \ref{derecha2} by taking into account that the function defined by 
\begin{equation*}
(x,y,t)\mapsto u(-x,y,1-t)
\end{equation*}
is also a solution of equation \eqref{ZK} satisfying the hypotheses of Lemma \ref{derecha2}. 

\newpage
\section{Estimates of Carleman type (Proof of Theorem \ref{Carleman})}

In the proof of the Carleman's estimate of Theorem \ref{Carleman} we will use the following Lemma:

\begin{lemma}\label{L2}
Let $w\in C^{1}( [ 0,1] ;L^{2}(\mathbb{R}^{2}) ) $ be a function such that for all $\beta >0$, $w$ is bounded from $[ 0,1] $ with values in $L^{2}( e^{2\beta |x| }e^{2\beta | y| }dxdy) $ and $w'\in L^{1}( [ 0,1] ;L^{2}( e^{2\beta|x| }e^{2\beta | y| }dxdy)) $. Then, for all  $\lambda \in \mathbb{R}$ and all $\xi =( \xi _{1},\xi _{2}) \in \mathbb{R}^{2}$, the functions $t\mapsto \widehat{e^{\lambda x}w( t) }( \xi )$ and $t\mapsto \widehat{e^{\lambda y}w( t) }( \xi )$ are absolutely continuous in $[ 0,1] $ with derivatives $\widehat{e^{\lambda x}w'( t)}( \xi ) $ and $\widehat{e^{\lambda y}w'( t) }( \xi ) $ $a.e.$ $t\in [ 0,1] ,$ respectively.
\end{lemma}

\begin{proof}

By symmetry, it is sufficient to prove the Lemma only for the weight $e^{\lambda x}$. Using Cauchy-Schwarz inequality, it is easy to see that for all $t\in[0,1]$ and $\lambda \in\mathbb R$, $e^{\lambda x}w( t) \in L^{1}( \mathbb{R}^{2})$, and also that $e^{\lambda x} w' \in L^1 (\mathbb{R}^2 \times [0,1])$ for all $\lambda\in \mathbb R$.

For $R>0$, let $\chi_R$ be the characteristic function of the square $[ -R,R] \times [ -R,R] $. 
Since $w\in C^{1}( [ 0,1] ;L^{2}( \mathbb{R}^{2}) ) $, the function
\begin{equation}
t\mapsto\int\limits_{\mathbb{R}^{2}}e^{-ix\xi _{1}}e^{-iy\xi _{2}}e^{\lambda x}\chi _R( x,y) w( t) (
x,y) dxdy=\left\langle w( t) ,e^{ix\xi _{1}}e^{iy\xi_{2}}e^{\lambda x}\chi _R\right\rangle _{L^{2}( \mathbb{R}^{2}) }
\label{34.1}
\end{equation}
defines a $C^1$ function of the variable $t$ with derivative given by \[
t\mapsto \left\langle w'( t) ,e^{ix\xi _{1}}e^{iy\xi_{2}}e^{\lambda x}\chi _R\right\rangle _{L^{2}( \mathbb{R}^{2}) },
\]
and in consequence
\begin{eqnarray}
\int\limits_{\mathbb{R}^{2}}e^{-ix\xi _{1}}e^{-iy\xi _{2}}e^{\lambda x}\chi _R( x,y) w( t) (
x,y) dxdy &=&\int_{0}^{t}\int\limits_{\mathbb{R}^{2}}e^{-ix\xi _{1}}e^{-iy\xi _{2}}e^{\lambda x}\chi _R( x,y) w'( \tau )( x,y) dxdyd\tau  \nonumber \\
&&+\int\limits_{\mathbb{R}^{2}}e^{-ix\xi _{1}}e^{-iy\xi _{2}}e^{\lambda x}\chi _R( x,y) w( 0) (
x,y) dxdy.  \notag
\end{eqnarray}
The Lemma follows from the former equality by an application of the Lebesgue Dominated Convergence Theorem. \end{proof}

{\it Proof of Theorem \ref{Carleman}:}
\begin{proof}
Let us define $g( t) :=e^{\lambda x}w( t) $ and $h( t) :=e^{\lambda x}( w'( t) +\partial_{x}^{3}w( t) +\partial _{x}\partial _{y}^{2}w( t))$. Then
\begin{equation}
h(t)=e^{\lambda x}w'(t)-\lambda ^{3}g(t)+3\lambda ^{2}\partial_{x}g(t)-3\lambda \partial _{x}^{2}g(t)+\partial _{x}^{3}g(t)-\lambda \partial_{y}^{2}g(t)+\partial _{x}\partial _{y}^{2}g(t).  \label{h}
\end{equation}

From the hypotheses on $w$ it can be seen that all terms in  \eqref{h} are in  $L^{1}( \mathbb{R}^{2})$ for almost every $t\in[0,1]$.
We take the spatial Fourier transform in \eqref{h} and apply Lemma \ref{L2} to obtain that
\begin{equation}
\frac{d}{dt}\widehat{g(t) }( \xi ) +[ -im( \xi ) -a( \xi )] \widehat{g( t) }( \xi )=\widehat{h( t) }( \xi ) ,\quad a.e. \;t\in[0,1],\label{ordinary}
\end{equation}
where
\[
m( \xi ) :=-3\lambda ^{2}\xi _{1}+\xi _{1}^{3}+\xi_{1}\xi _{2}^{2},\quad\text{and}\quad
a( \xi ) :=\lambda ^{3}-3\lambda \xi _{1}^{2}-\lambda\xi _{2}^{2}.
\]
Using \eqref{ordinary}, when $a(\xi)\leq0$, we  have
\[
\widehat{g( t) }( \xi ) =e^{im(\xi) t}e^{a( \xi ) t}\widehat{g( 0) }
( \xi ) +\int_{0}^{t}e^{im( \xi ) (t-\tau) }e^{a( \xi ) (t-\tau)}\widehat{h( \tau ) }( \xi )d\tau , \quad\text{for all }t\in[0,1],
\]
and when $a(\xi)>0$, we choose to write
\[
\widehat{g( t) }( \xi ) =e^{-im( \xi) ( 1-t) }e^{-a( \xi ) (1-t) }\widehat{g( 1) }( \xi )-\int_{t}^{1}e^{-im( \xi ) ( \tau -t)}e^{-a( \xi ) ( \tau -t) }\widehat{h(\tau ) }( \xi ) d\tau  \quad\text{for all }t\in[0,1].
\]
In any case, for all $t\in[0,1]$:
\[
| \widehat{g( t) }( \xi ) | \leq  | \widehat{g( 0) }( \xi ) |+| \widehat{g( 1) }( \xi ) |+\int_{0}^{1}| \widehat{h( \tau ) }( \xi )| d\tau ,
\]
and estimate \eqref{EstimativoCarleman} follows from Plancherel's formula.

The proof of the estimate with the weight $e^{\lambda y}$ is similar.\end{proof}

\section{Proof of Theorem \ref{main}}
\begin{proof}
Let $\widetilde{\phi }\in C^{\infty }( \mathbb{R})$ be a non-decreasing function such that $\widetilde{\phi}( x) =0$ for
$x<0$ and $\widetilde{\phi }( x) =1$ for $x>1$ and, for $R>B$,  let $\phi (x) \equiv \phi _{R}( x):=\widetilde{\phi }(x-R) $ .
We define $w\equiv w_R:= \phi(x)u$, and $v\equiv v_R:=\phi(y)u$. Since $supp\,u(0)$ and $supp\,u(1)$ are compact, from Theorem \ref{decay} and equation \eqref{ZK}, it follows that $w$ and $v$ satisfy the hypotheses of Theorem \ref{Carleman}. 

Taking into account that $w(0)=w(1)=0$, from \eqref{EstimativoCarleman} we conclude that
\begin{align}
\| e^{\lambda x}w\|&\leq \| e^{\lambda x}(w'+\partial _{x}^{3}w+\partial _{x}\partial_{y}^{2}w) \|   \nonumber \\
&=\| e^{\lambda x}( \phi u'+\phi \partial_{x}^{3}u+\phi \partial _{x}\partial _{y}^{2}u+\phi '''u+3\phi ''\partial _{x}u \nonumber +3\phi '\partial _{x}^{2}u+\phi '\partial_{y}^{2}u) \|   \nonumber \\
&\leq \| e^{\lambda x}\phi u\partial _{x}u\| +\| e^{\lambda x}F_{1\phi,u}\|,  \label{42}
\end{align}
where $\phi:=\phi(x)$, $\|\cdot\|:=\|\cdot\|_{L^{2}( \mathbb{R}^{2}\times [ 0,1] ) }$ and
\[
F_{1\phi ,u}:=\phi '''u+3\phi ''\partial_{x}u+3\phi '\partial _{x}^{2}u+\phi '\partial _{y}^{2}u.
\]

Since the derivatives of $\phi$ are supported in the interval $[R,R+1]$, it can be seen that
\begin{equation}
\| e^{\lambda x}F_{1\phi ,u}\|\leq Ce^{\lambda ( R+1)}.
\label{43}
\end{equation}
where $C=C(\| u\|_{C([0,1];H^2)})$, and is independent from $\lambda$ and $R$.
Therefore
\[\| e^{\lambda x}\phi u\|\leq \| e^{\lambda x}\phi u\|\|\partial_x u\|_{L^\infty([R,+\infty)\times\mathbb R\times[0,1])}+ Ce^{\lambda ( R+1)}.
\]
From Theorem \ref{decay}, with $\beta=1$ and Sobolev imbeddings, there exists a constant $C_1$ such that 
\[
| \partial _{x}u( t) (x,y) | \leq C_{1}e^{-x }.
\]
Thus
\begin{equation}
\| e^{\lambda x}\phi u\|\leq C_{1}e^{-R}\| e^{\lambda
x}\phi u\| +Ce^{\lambda ( R+1) }.
\label{ad1}
\end{equation}
Since, from Lemma \ref{derecha} $\|e^{\lambda x}\phi u \|<\infty$, we can absorb the first term on the right hand side of \eqref{ad1} by taking $R>B$ such that $C_1e^{-R}<\frac12$ to obtain that
\begin{equation*}
\| e^{\lambda x}\phi u\|\leq Ce^{\lambda ( R+1) }.
\end{equation*}
And thus, since $\phi(x)=1$ for $x\geq 2R$,
\begin{equation}
e^{2\lambda R}( \int_{0}^{1}\int_{-\infty}^{\infty}\int_{2R}^{\infty }| u( t) ( x,y)| ^{2}dxdydt) ^{1/2}\leq\| e^{\lambda x}\phi u\|\leq \| e^{\lambda x}\phi u\|\leq Ce^{\lambda(R+1)}  \label{46}
\end{equation}
Since \eqref{46} is valid for all $\lambda>0$, $2R>R+1$, and the constant $C$ is independent from $\lambda$, by letting $\lambda\to+\infty$ it follows that
\[
( \int_{0}^{1}\int_{-\infty}^{\infty}\int_{2R}^{\infty }| u( t) ( x,y)| ^{2}dxdydt) ^{1/2}=0.
\]
Thus $u\equiv 0$ in $[ 2R,\infty ) \times\mathbb{R}\times [ 0,1] .$

In a similar way, for $v:=\phi(y)u$, taking into account that $v(0)=v(1)=0$ , an application of the Carleman's estimate \eqref{EstimativoCarleman} with weight $e^{\lambda y}$ gives:
\begin{align*}
\| e^{\lambda y}\phi u\|=\|e^{\lambda y}v\| \leq &\| e^{\lambda y}(v'+\partial _{x}^{3}v+\partial _{x}\partial_{y}^{2}v) \|\\
=&\| e^{\lambda y}(\phi u'+\phi \partial _{x}^{3}u+\phi\partial _{x}\partial _{y}^{2}u+2\phi '\partial _{x}\partial_{y}u+\phi ''\partial _{x}u)\|\\
\leq &\| e^{\lambda y}\phi u\partial _{x}u\|+\| e^{\lambda y}F_{2\phi,u}\|,
\end{align*}
where
\[
F_{2\phi ,u}:=2\phi '\partial _{x}\partial _{y}u+\phi ''\partial _{x}u.
\]
Now we reason as above to conclude that $u\equiv 0$ in $\mathbb{R}\times[ 2R,\infty )\times [ 0,1] $.

Finally, we notice that the function $(x,y,t)\mapsto u(-x,-y,1-t)$ also satisfies the hypotheses of Theorem \eqref{main}, which, by the former procedure, implies that $u\equiv 0$ in  $ (-\infty, -2R ]\times\mathbb{R}\times [ 0,1] \,\cup\,\mathbb{R}\times( -\infty,-2R]\times [ 0,1] $.

In this manner, there exists $R>0$ such that $supp\,u(t)\subset [-2R,2R]\times[-2R,2R]$ for all $t\in[0,1]$. Then, by Theorem 
\ref{Panthee}, $u\equiv 0$.
\end{proof}

\end{document}